\documentclass[a4paper,12]{elsarticle}

\usepackage[utf8]{inputenc}
\usepackage{times} %PostScript Times Roman instead of Computer Times Roman.
\usepackage{latexsym,amsfonts,amssymb,amsthm,amsmath,amscd,mathrsfs}
\usepackage{xargs}   % Use more than one optional parameter in a new commands
\usepackage{xspace}
\usepackage[shortlabels]{enumitem}
\usepackage{graphicx}
\usepackage{float}
\usepackage{multirow}
\usepackage[tight]{subfigure}
\usepackage{hyperref}
\usepackage{tikz}
\usepackage{pgfplots}
\usepackage{tabulary}
\usepackage{tabularx}
\pgfplotsset{compat=newest} 
\pgfplotsset{plot coordinates/math parser=false} 
\newlength\figureheight 
\newlength\figurewidth 

\usepackage{xpatch}
\usepackage{algpseudocode}
\usepackage{algorithm}
\newtheorem{theorem}{Theorem}

\newtheorem{definition}[theorem]{Definition}
\newtheorem{proposition}[theorem]{Proposition}

\newtheorem{property}[theorem]{Property}
\newtheorem{remark}[theorem]{Remark}

\newtheorem{assumption}[theorem]{Assumption}
\newcommand{\vsig}{\mbox{\boldmath$\sigma$}}
\newcommand{\valpha}{\mbox{\boldmath$\alpha$}}
\newcommand{\vx}{\mathbf{x}}

\newcommand{\R}{\ensuremath{\mathbb{R}}\xspace} 
\newcommand{\N}{\ensuremath{\mathbb{N}}\xspace} 
\newcommand{\X}{\ensuremath{\mathcal{X}}\xspace}

\newcommand{\Z}{\ensuremath{\mathcal{Z}}\xspace}

\DeclareMathOperator{\inti}{int}

\usepackage[shadow,colorinlistoftodos,prependcaption,spanish,textsize=footnotesize]{todonotes}
\newcommandx{\falta}[2][1=]{\todo[linecolor=red,backgroundcolor=red!25,bordercolor=red,#1]{#2}}
\newcommandx{\completar}[2][1=]{\todo[linecolor=blue,backgroundcolor=blue!25,bordercolor=blue,#1]{#2}}
\newcommandx{\chequear}[2][1=]{\todo[linecolor=OliveGreen,backgroundcolor=OliveGreen!25,bordercolor=OliveGreen,#1]{#2}}
\newcommandx{\improve}[2][1=]{\todo[linecolor=Plum,backgroundcolor=Plum!25,bordercolor=Plum,#1]{#2}}
\newcommandx{\explain}[2][1=]{\todo[linecolor=lime,backgroundcolor=lime!25,bordercolor=lime,#1]{#2}}
\numberwithin{equation}{section}
\numberwithin{theorem}{section}

\allowdisplaybreaks
\begin{document}
\begin{frontmatter}

\title{Discrete-time MPC for switched systems with applications to biomedical problems}

\author[First]{A. Anderson}
\author[First]{A. H. Gonz\'alez}
\cortext[corr-author]{Corresponding authors}
\ead{alejgon@santafe-conicet.gov.ar}
\author[Second]{A. Ferramosca}
\author[Third]{E. Hernandez-Vargas}
\ead{vargas@fias.uni-frankfurt.de}

\address[First]{Institute of Technological Development for the Chemical Industry (INTEC),\\ CONICET-Universidad Nacional del Litoral (UNL). G\"uemes 3450, (3000) Santa Fe, Argentina.}

\address[Second]{CONICET - Universidad Tecnol\'ogica Nacional (UTN). Facultad Regional de Reconquista. Calle 27 de Abril, 1000 (3560), Reconquista, Santa Fe, Argentina.}

\address[Third]{Institute of Mathematics, UNAM, Mexico. Research Fellow, Frankfurt Institute for Advanced Studies, Germany.}

%\footnote{Corresponding author: alejgon@santafe-conicet.gov.ar, vargas@fias.uni-frankfurt.de}

%%%%%%%%%%%%%%%%%%%%%%%%%%%%%%%%%%%%%%%%%%%%%%%%%%%%%%%%%%%%%%%%%%%%%%%%%%%%%%%%
\begin{abstract}
Switched systems in which the manipulated control action is the time-depending switching
signal describe many engineering problems, mainly related to biomedical applications. In such a context, to control the system 
means to select an autonomous system - at each time step - among a given finite family. Even when this selection can be done by solving a Dynamic 
Programming (DP) problem, such a solution is often difficult to apply, and state/control constraints cannot be explicitly considered.
In this work a new set-based Model Predictive Control (MPC) strategy is proposed to handle switched systems in a tractable form. 
The optimization problem at the core of the MPC formulation consists in an easy-to-solve mixed-integer optimization problem, whose 
solution is applied in a receding horizon way. Two biomedical applications are simulated to test the controller: 
(i) the drug schedule to attenuate the effect of viral mutation and drugs resistance on the viral load, and 
(ii) the drug schedule for Triple Negative breast cancer treatment.
The numerical results suggest that the proposed strategy outperform the schedule for available treatments.
\end{abstract}
\end{frontmatter}

\section{Introduction}
{Many engineering problems related to the control of mechanical, automotive, power, aircrafts, traffic, and biomedical 
	dynamical systems can be properly described by the so-called switched systems.
	Switched systems are dynamical systems consisting in a collection of subsystem and a rule (switching signal) 
	that governs the switching among them.
	The dynamical subsystems of the collection can be autonomous (uncontrolled) or controlled, while the switching signal 
	can be state-dependent or time-dependent \cite{Liberzon03}. In this work the interest is put on the autonomous, time-dependent case.
	In this context, the switching signal may be regarded as the manipulated control action (the decision variable), in such a way that the control objective
	is achieved (exclusively) by a proper combination of subsystems trough the time.}
A {particularly} interesting application of {such kind of} switched systems is the problem of scheduling 
therapies {in} different biomedical problems \cite{ChapmanCDC18,Vargasbook19}.
In \cite{Vargasbook19}{, for instance,} the {control problem of minimizing the viral load while 
	delaying the emergence of highly resistant mutant viruses \cite{Clavel04} is studied, by means of a switching strategy. Each autonomous submodel
	represents the virus dynamic under a given specific treatment (specified by a given drug or combination of drugs), and the switching
	signal decides which treatment is used at each time.} The clinical goal is to delay the time until patients exhibit resistant strains to all existing 
{drug} regimens \cite{Martinez08}. {One of the challenging problem is the} crucial trade-off between switching therapies: on the one hand, switching early carries the risk of poor adherence to a new drug regimen and prematurely exhausting the limited number of remaining therapies; on the other hand, switching drugs too late allows the accumulation of mutations that leads to multi-drug resistance \cite{Molla96}. 
In \cite{ChapmanCDC18} a switched linear system of the response of the Triple Negative breast cancer cell line to different treatment 
{is} studied. A cyclic drug schedule {is} proposed based on the {hypotheses that 
	some drugs} can shrink the live cancer cell population, but they are very toxic to healthy cells \cite{RisomNat18} (transient rather than long-term efficacy, low rate of response, negative secondary reactions \cite{Migliardi12,Jokinen15,Grilley16}){. On the other hand,
	other drugs only slow down} the growth of the live cancer cell population{, but} are less toxic to healthy cells. 

{Concerning the techniques able to solve the previous problem, Dynamic Programming, (DP) \cite{bertsekas1995dynamic} 
	is a first strategy suitable to compute the optimal 
	sequence of subsystem, for a given switched system problem. However, the DP solution is often difficult to apply (impossible in most of the cases), 
	and state/control constraints cannot be explicitly considered.
	One step ahead is the Model Predictive Control (MPC) strategy, which by means of the Receding Horizon Control (RHC) policy overcomes all the 
	implementation problems of DP, as stated in \cite{rawlings2017model}.}
The main features of MPC - which make it one of the most employed advanced control technique - are the explicit consideration of a 
model for prediction, the optimal computation of the control {actions}, and its ability to handle, easily and effectively, 
hard constraints on control and states \cite{MayneAUT00,rawlings2017model}. MPC theoretical background has been widely investigated in the last 
decades, showing how this technique is capable to provide stability, robustness, constraint satisfaction and tractable computation for linear and 
nonlinear systems \cite{RawlingsLIB09,AndersonOCAM18}. Set invariance theory, which is closely related to Lyapunov stability theory \cite{Blanchinibook15}, 
has also shown to be a powerful tool for analyzing dynamical systems subject to constraints. {Set-based MPC \cite{AndersonSCL18}} 
seems to be an appropriate strategy to undertake {such a control problems}.

The application of the MPC to switched systems is a growing field; the switching law is in fact either considered as a perturbation \cite{SunCCE11} 
or as part of the control inputs. In this last case, conditions for stabilizability have been provided by using a min-switching policy \cite{Liberzon03}, 
and Lyapunov–Metzler inequalities \cite{GeromelIJC06}. In this context, the stability of switched systems is neither intuitive nor trivial: 
switching between individually stable sub-systems may cause instability and conversely, switching between unstable sub-systems may yield a stable 
switched system \cite{Liberzon03}. The easy implementation and the anticipatory nature of the MPC seems an appropriate strategy for computing switching 
laws, since it may anticipate the activation of possible switching. Moreover, set-theory has been recently used in the context of stability of switching systems \cite{FiacchiniAUT14}, suggesting that set-based MPC approach {is} a promising tool for the analysis of stability, robustness and constraint satisfaction of this kind of systems.

This paper {studies} the behavior of a discrete-time switched system and a the invariance and controllable sets of the state space for this type of systems. A novel set-based MPC formulation for discrete-time switched linear systems is presented with asymptotic stability guaranteed. The performance of the proposal is assessed by several simulations in the biomedical fields, a viral mutation problem and a drug scheduling synthesis for cancer treatment is addressed showing that the proposed MPC outperformed the available treatments.

The paper is organized as follows, Section~\ref{Sec:D-TSwSys} introduces the switched linear system and an optimal control solution, Section~\ref{Sec_Invariance} presents an invariance set-theory for switched systems, in Section~\ref{sec:mainresult} the set-based MPC is formulated and the asymptotic stability is proved, Section~\ref{Sec_Application} address with several simulations the application to biomedical problems and Section~\ref{Sec:conclusion} states the conclusion of the work.

\subsection{Notation}
We denote, by $\mathbb N$ the set of natural numbers, by $\mathbb N_q=\{n\in\mathbb N:1\leq n\leq q \}$, by $\mathbb Z$ the set of integer numbers, by $\mathbb Z_{\geq 0}$ the set of non-negative integers, by $\mathbb Z_q=\{n\in\mathbb Z:0\leq n\leq q \}$, and by $\mathbb Z_{l:q}=\{n\in\mathbb Z:l\leq n\leq q \}$.
The Euclidean distance between two points $x,~y$ on $\mathbb{R}^n$ is represented by $d(x,y)$.
The distance from a point $x\in\mathbb{R}^n$ to a set $\Omega$ is defined as $d_{\Omega}(x) := \inf \{d(x,y):~y\in\Omega\}$. 
We define the cardinal number of a set $\Omega$ as $\mid \Omega\mid$.
The interior of a set $\Omega\subseteq\R^n$ is denoted by $\inti \Omega$.
A set $\Omega\subseteq \R^n$ is called star-convex with respect to the origin if for all $x\in\Omega$ the line segment from $0$ to $x$ is in $\Omega$.
A set $\Omega\subseteq \R^n$ is a $C^*$-set if it is compact, star-convex with respect to the origin and $0\in\inti \Omega$.

%%%%%%%%%%%%%%%%%%%%%%%%%%%%%%%%%%%%%%%%%%%%%%%%%%%%%%%%%%%%%%%%%%%%%%%

\section{Discrete-time linear switched system and optimal control solution}\label{Sec:D-TSwSys}

%%%%%%%%%%%%%%%%%%%%%%%%%%%%%%%%%%%%%%%%%%%%%%%%%%%%%%%%%%%%%%%%%%%%%%%
Consider the discrete-time switched system described by 
\begin{equation} \label{eq:SistOrig}
x(k+1) = A_{\sigma(k)}x(k),
\end{equation}
with $x(0)=x_0$, where $x(k) \in \X \subset \R^n$ is the system state at the $k$--th sample time, and the set $\X$ is closed. $\sigma(k)\in \Sigma:=\{1,2,\dots,q\}$ is the switching signal that, at any instant, selects indirectly a transition matrix $A_i\in\R^{n\times n}$ for $i\in\Sigma$.
In what follows the signal $\sigma(\cdot)$ is considered as manipulable variable. The problem of stabilizability consists in proving the existence of switching signals that yield asymptotic stability if applied to~\eqref{eq:SistOrig}.

In general, different initial conditions may correspond to different switching paths, because a switching path $\vsig=\{\sigma(0),\sigma(1),\dots,\sigma(T-1) \}$ depends heavily on the initial state $x_0$ \cite{Zhendong05}. This is a critical feature which makes the switched system essentially distinct from a linear time-varying system, where all initial states correspond to a single path (further details in~\cite[Example 1.5]{Zhendong05}).

Let $L_{\sigma}\in\N$ and $U_{\sigma}\in\N$ which depends on every signal $\sigma\in\Sigma$ with $L_{\sigma}\le U_{\sigma}$. The following assumption implies that every switching signal $\sigma$ must be applied to the system at least $L_{\sigma}$ instant of times and no more than $U_{\sigma}$ instant of times.

\begin{assumption}[Waiting times]\label{ass:UL}
Every signal $\sigma\in\Sigma$ must be applied to the switched system~\eqref{eq:SistOrig} at least $L_{\sigma}$ times and no more than $U_{\sigma}$ times.
\end{assumption}

Generally, to corner the above assumption in the bibliography it is considered the following sequence 
\begin{align}\label{eq_h_j}
\{(\sigma(0),h(0)),(\sigma(1),h(1)),\dots,(\sigma(T-1),h(T-1))\}
\end{align}
where $h(j)$ is an optimization variable that represents the number of times that signal $\sigma(j)$ is applied to the system, for $j\in\mathbb{Z}_{T-1}$. With this arrangement Assumption~\ref{ass:UL} hold with
\[L_{\sigma(j)}\leq h(j) \leq U_{\sigma(j)}\] for all $j\in\mathbb{Z}_{T-1}$.

In what follows, different notation will be used in such a way there is no need to add another variable, $h(j)$, to the optimization problem. 

Consider the following definition.
\begin{definition}[$j$-pack]\label{def_pack}
	Given a switching path $\vsig=\{\sigma(0),\sigma(1),\dots,\sigma(T-1) \}$. For every time $j\in\mathbb Z_{T-1}$ we define the $j$-pack set associated with $\vsig$ by
	\begin{align}\label{eq:quantum}
	\mathcal{P}[\vsig,j]:=\{\sigma(k+j): \text{ with }k\in\mathbb Z_{-j:T-1-j} \text{ such that } \sigma(j)=\sigma(j+i)~\forall~i=0,\pm1,\dots,k \}.
	\end{align}
	The term $i=0,\pm1,\dots,k$ in~\eqref{eq:quantum} depends on the sign of $k$, that is
	\begin{itemize}
		\item[] if $k\geq 0 \Rightarrow i= 0,1,\dots,k$ 
		\item[] if $k<0 \Rightarrow i= 0,-1,\dots,k$.
	\end{itemize} 
\end{definition}

The set $\mathcal{P}[\vsig,j]$ is composed by all signals in $\vsig$ consecutively identical to the signal $\sigma(j)$. According to Eq.~\eqref{eq:quantum},
\begin{itemize}
	\item $\sigma(j)\in\mathcal{P}[\vsig,j]$ for all $j\in \mathbb Z_{T-1}$. 
	\item If $\sigma(k+j)\in\mathcal{P}[\vsig,j]$ for some $k\in\mathbb Z_{0:T-1-j}$ $\Rightarrow$ $\sigma(k+j)=\sigma(j)$, moreover  \[\sigma(j)=\sigma(j+1)=\sigma(j+2)=\dots=\sigma(j+k), \] i.e., $\sigma(j+1)\in\mathcal{P}[\vsig,j],~\sigma(j+2)\in\mathcal{P}[\vsig,j],\dots,\sigma(j+k)\in\mathcal{P}[\vsig,j]$.
	\item If $\sigma(k+j)\in\mathcal{P}[\vsig,j]$ for some $k\in\mathbb Z_{-j:0}$, $\Rightarrow$ $\sigma(k+j)=\sigma(j)$, moreover  \[\sigma(j)=\sigma(j-1)=\sigma(j-2)=\dots=\sigma(j+k), \] i.e., $\sigma(j-1)\in\mathcal{P}[\vsig,j],~\sigma(j-2)\in\mathcal{P}[\vsig,j],\dots,\sigma(j+k)\in\mathcal{P}[\vsig,j]$.  
\item For instance, if $\vsig=\{1,2,2,2,3,3,2\}$, then 
$\mathcal{P}[\vsig,0]=\{1 \},$
$\mathcal{P}[\vsig,1]=\{2,2,2 \},$
$\mathcal{P}[\vsig,2]=\{2,2,2 \},$
$\mathcal{P}[\vsig,3]=\{2,2,2 \},$
$\mathcal{P}[\vsig,4]=\{3,3 \},$
$\mathcal{P}[\vsig,5]=\{3,3 \},$
$\mathcal{P}[\vsig,6]=\{2 \}.$\\
\end{itemize}
This way, variable $h(j)$ from Eq.\eqref{eq_h_j} is not necessary to address Assumption~\ref{ass:UL}. Instead the following property will be considered, 

\begin{property}\label{prop:dwelltime}
	A switching path $\vsig=\{\sigma(0),\sigma(1),\dots,\sigma(T-1) \}$ fulfills Assumption~\ref{ass:UL} if 
	\[L_{\sigma(j)}\leq\mid\mathcal{P}[\vsig,j]\mid \leq U_{\sigma(j)}\] for all $j\in\mathbb{Z}_{0:T-1}$.
\end{property}

%%%%%%%%%%%%%%%%%%%%%%%%%%%%%%%%%%%%%%%%%%%%%%%%%%%%%%%%%%%%%%%%%%%%%%%
\subsection{Optimal solution for discrete-time positive switched system}\label{Sec:OptimalSolution}

In this section we introduce the optimal control for the discrete-time positive switched system~\eqref{eq:SistOrig}. 
In order to stabilize the origin, we consider the cost function
\begin{align}\label{eq:CostOrigin}
\mathcal{J}_N(x,\vsig) = \sum_{k=0}^{N-1}c_{\sigma(k)} x(k)+c x(N)
\end{align}
where $x=x(0)$ is the current state; $x(j+1)=A_{\sigma(j)}x(j)$, for $j \in \mathbb Z_{N-1}$;  $\vsig$ is a switching discrete path, $c_{\sigma(j)}$ is a positive weight vector corresponding to signal $\sigma(j)$, and vector $c$ is a positive final weight.
Cost function~\eqref{eq:CostOrigin} must satisfy the Hamilton–Jacobi–Bellman equation \cite{Locatelli01}. 

If we define the optimal switching signal, the corresponding trajectory and the optimal cost functional as $\sigma^0(k)$, $x^0(k)$ and $\mathcal{J}_N(x,\vsig^0)$, respectively, where $\vsig^0:=\{\sigma^0(0),\dots,\sigma^0(N-1) \}$. 
Using the Hamilton–Jacobi–Bellman equation for the discrete case, we have:
\[\mathcal{V}(x(k),k)=\min_{\sigma(k)\in\Sigma}\{c_{\sigma(k)}+\mathcal{V}(x(k+1),k+1)\}.\]
The general solution for this system is given by 
\[ \mathcal{V}(x(k),k)=p(k)'x(k),\]
where $p(k)$ denote the costate vector. Therefore, the following nonlinear system is obtained:
\begin{align}\label{eq:optimal}
&x^0(k+1)=A_{\sigma^0(k)}x^0(k),~~x(0)=x_0 \\
&p^0(k)=A_{\sigma^0(k)}'p^0(k+1)+c_{\sigma^0(k)},~~p(T)=c \notag\\
&\sigma^0(k)=arg~ \min_s\{p^0(k+1)'A_sx^0(k)+c_sx^0(k). \} \notag
\end{align}

The state equation is subject to an initial condition and is solved forwards in time, whereas the costate equation must be integrated backward, both according to the coupling condition
given by the switching rule. As a result, the problem is a two-point boundary value problem, and cannot be solved using regular iteration techniques. 
%
%In \cite{hernandez2011discrete}, 
%under the assumption of system positivity and considering only a final linear cost function (i.e., $\mathcal{J}_N(x,\sigma)=cx(N)$), a Dynamic 
%Programming technique is proposed to find the exact solution. At each time step of the system evolution, a feedback law $u(x(k))$ can be 
%obtained, by solving a Linear Programming (LP) problem.

However, as detailed in \cite{hernandez2011discrete}, obtaining the aforementioned solution could be difficult, if not impossible, because of the computational complexity.
In the next section, a Receding Horizon (RHC) strategy will be presented that - although sub-optimal - reasonably approximates the optimal solution, 
at a significant smaller computational cost. Furthermore, the proposed strategy allows to consider a complete cost function (penalizing the states 
all along a given horizon) and considers full state constraints.

\section{Invariance for switched systems}\label{Sec_Invariance}
The set-theory is an important tool for the analysis of dynamic systems. In particular, there is no consensus for the characterization of invariant sets for switched systems. 
In what follows, we present an analysis of invariance for switched systems that follows the general concept of classical invariance from \cite{Blanchinibook15}. Additionally, sufficient conditions for its existence is presented.
\begin{definition}[Switched invariant set]\label{def:SIS} 
	A set $\Omega \subset \X$ is a switched invariant set (SIS) of system \eqref{eq:SistOrig} if for all $x \in \Omega$, there exists $\sigma \in \Sigma$ such that $A_{\sigma}x \in \Omega$.
	%	%
	%	Associated to $\Omega$ is the corresponding input set \[\Psi(\Omega):= \{ \sigma \in \mathbb Z_{1:D} : \exists ~ x \in \Omega \text{ such that } A_{\sigma}x \in \Omega \}.\]
\end{definition}

In this section we suppose the following assumption.

\begin{assumption}
	The matrices $A_i$, with $i \in\Sigma$ are nonsingular.
\end{assumption}

\begin{definition}[Controllable set]\label{def:CS} 
	Given a set $\Omega \subset \X$, the controllable set to $\Omega$, $S(\Omega,\Sigma)$, is defined by
	\begin{equation}\label{eq:union}
	S(\Omega,\Sigma):=\bigcup_{i\in \Sigma} A^{-1}_i\Omega.
	\end{equation}
\end{definition}
As usual, the characterization of the invariance can be done with the analysis of the controllable set, i.e. a set is an invariant set if and only if it is contained in its controllable set \cite{Blanchinibook15}. In the context of switched systems we have the equivalent proposition:
\begin{proposition}\label{prop:invariant}
	Let the compact and convex set $\Omega\subset\X$ such that
	$\Omega \subseteq S(\Omega,\Sigma)$, then $\Omega$ is a SIS of system~\eqref{eq:SistOrig}.
\end{proposition}
\begin{proof}
	Let $x\in\Omega$, then by equation \eqref{eq:union}, there exists $i\in \Sigma$ such that $x\in A^{-1}_i\Omega$. Therefore $A_{i}x\in\Omega$.
\end{proof}
Figure~\ref{fig:SIS} shows a set $\Omega$, such that $\Omega\subseteq \bigcup_{i\in \Sigma} A^{-1}_i\Omega$ but $\Omega \not\subseteq A^{-1}_i\Omega$ for any $i=1,2,3$, meaning that $\Omega$ is an invariant set for the switched system $x(k+1)=A_{\sigma(k)}x(k)$ with $\sigma(k)\in\Sigma=\{1,2,3\}$, but is not an invariant set for neither of the autonomous linear sub-systems $x(k+1)=A_ix(k)$ for $i=1,2,3$. This points to the fact that a switched system can be stable even if it is composed by unstable sub-systems \cite{Liberzon03}.

\begin{figure}
	\centering
	\includegraphics[width=0.65\textwidth]{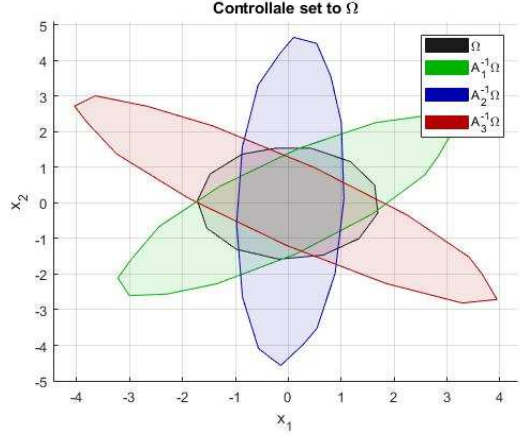}
	\caption{A switched invariant set, $\Omega$, according to Proposition~\ref{prop:invariant}. Figure shows that $\Omega\subset S(\Omega,\Sigma)$.}
	\label{fig:SIS}
\end{figure}

\begin{definition}[$i$-Step controllable set]\label{def:iSCS} 
	Given a set $\Omega \subset \X$, the $i$-step controllable set to $\Omega$, $S_i(\Omega,\Sigma)$, is defined by the following algorithm:
	\begin{enumerate}
		%		\item Let $S_i:=S_i(\Omega,\Sigma)$, for $i=1,\dots,i$
		\item Initialization: $S_1(\Omega,\Sigma):=S(\Omega,\Sigma)$.
		\item Iteration: $S_i(\Omega,\Sigma):=S(S_{i-1}(\Omega,\Sigma),\Sigma)$.
	\end{enumerate}
\end{definition}
%In general, the $i$-step controllable sets are not convex sets, but if the set $\Omega$ is a convex set, then they are star-convex sets, i.e. a controllable set $S_i(\Omega,\Sigma)$ is star-convex if there exists $x^0\in S_i(\Omega,\Sigma)$ such that every convex combination of $x$ and $x^0$ belongs to $S_i(\Omega,\Sigma)$ for every $x\in S_i(\Omega,\Sigma)$.

\begin{proposition}(\cite[Algorithm 1]{FiacchiniAUT14})
	Let $\Omega \subseteq \R^n$ be a $C^*$-set, if there exists $k\in\mathbb Z_{\ge 0}$ such that \[\Omega\subseteq \inti \bigcup_{j\in\N_{k+1}} S_j(\Omega,\Sigma),\] then there is a switching law stabilizing the switched system~\eqref{eq:SistOrig}.
\end{proposition}

\begin{proposition}(\cite[Algorithm 3]{FiacchiniAUT14})
Let $\Omega \subseteq \R^n$ be a $C^*$-set, and define the set \[\hat{\Omega}_{k+1}:=\bigcup_{j\in\N_{k+1}}S_j(\Omega,\Sigma)\cup \Omega, \] 
	if there exists $k\in\mathbb Z_{\ge 0}$ such that \[ S_{k+1}(\Omega,\Sigma)\subseteq \hat{\Omega}_k,\] then there is no switching law stabilizing the switched system~\eqref{eq:SistOrig}.
\end{proposition}

\section{Switching set-based MPC (SwMPC)}\label{sec:mainresult}
This section introduces a proper set-based MPC for switched systems with guaranteed asymptotic stability of the closed-loop system.\\

For a given (fixed) horizon $N\in\mathbb{N}$, and a target set $\Omega\subseteq \X$, compact, convex and with the origin in its interior, the following (set-dependent) cost function is proposed:
\begin{eqnarray} \label{Cost}
J_N(x;\vsig):= \sum\limits_{j=0}^{N-1} c_{\sigma(j)}d_{\Omega}(x(j))+cd_{\Omega}(x(N))
\end{eqnarray}
where $x=x(0)$ is the current state; $x(j+1)=A_{\sigma(j)}x(j)$, for $j \in \mathbb Z_{N-1}$;  $\vsig$ is a switching discrete path, $c_{\sigma(j)}$ is a positive weight vector corresponding to signal $\sigma(j)$, and vector $c$ is a positive final weight.
%

%\subsubsection{Proposed formulation}
Let us consider a binary variable $\alpha_i^j\in\{ 0,1\}$, for all $i\in\Sigma$ and $j\in\mathbb Z_{N-1}$, such that: 
\begin{itemize}
	\item $\alpha_i^j=1 \Rightarrow \sigma(j)=i$
	\item $\alpha_i^j=0 \Rightarrow \sigma(j)\not=i$.
\end{itemize}
Then, if we consider the set of integer optimization variables $\valpha=\{\alpha_i^j,~ i\in\Sigma, j\in\mathbb Z_{N-1} \}$, the sequence of signals $\vsig$ can be obtained by the matrix $\valpha$.

Consider a target set $\Omega\in\X$, and the initial state $x$ at time $k$. The optimization problem is defined as follow:
\begin{align} \label{EqProblem} 
\min_{\valpha} &~~J_N(x;\vsig(\valpha)) \\
\text{s.t. }\quad 
& x(0) = x,\label{EqProblem0}\\ 
& x(j+1)=\sum_{i=1}^{q}\alpha_i^jA_{\sigma(j)}x(j),~~~~~ j \in \mathbb Z_{N-1},\label{EqProblem1}\\
& \alpha_i^j\in\{0,1\},~~~~~~~~~~~~~~~~~~~~~~~~~~~~~~~~j \in \mathbb Z_{N-1},~~~ i \in \Sigma,\label{EqProblem2}\\
&\sum_{i=1}^{q}\alpha_i^j=1,~~~~~~~~~~~~~~~~~~~~~~~~~~~~~~~~ j \in \mathbb Z_{N-1},\label{EqProblem3}  \\
& x(j) \in \X,~~~~~~~~~~~~~~~~~~~~~~~~~~~~~~~~~~~ j \in \mathbb Z_{N-1} \label{EqProblem4}  \\
&\sigma(j)=\{i:\alpha_i^j=1 \},~~~~~~~~~~~~~~~~~j\in\mathbb Z_{N-1}\label{EqProblem5}\\
%&a\leq\mid \mathcal{P}[\sigma(j)]\mid\leq b,~~~~~~~~~~~~~~~~~~~~j\in\mathbb Z_{-a:N-1}, \label{EqProblem6}  \\
& x(N) \in\Omega. \label{EqProblem7}
\end{align}
Eq.~\eqref{EqProblem0} is the initialization of the problem. Eq.~\ref{EqProblem1} the evolution of the switched system in terms of the variable $\alpha_i^j$. Eq.~\ref{EqProblem2} the integer optimization variables. Eq.~\ref{EqProblem3} means that only one signal is applied in every step $j$. Eq.~\ref{EqProblem5} relates signal $\sigma(j)$ with the optimization variable $\alpha_i^j$, and Eq.~\ref{EqProblem7} is the final constraint which implies that the final predicted state belongs to the target set $\Omega$.

Problem~\eqref{EqProblem} does not necessarily meet the waiting time of Assumption~\ref{ass:UL}. To fulfill Assumption~\ref{ass:UL} we use Property~\ref{prop:dwelltime}. First, let us state that 
\[U=\max_{\sigma\in\Sigma}\{U_{\sigma}\}.\]
We define a memory switching path $\vsig_m$ by
\[ \vsig_m:=\{ \underbrace{\sigma^0(-U),\sigma^0(-U+1),\dots,\sigma^0(-1)}_{memory},\underbrace{\sigma(0),\sigma(1),\dots,\sigma(N-1)}_{predictions} \}, \]
where the first $U$ elements of $\vsig_m$ are the last $U$ optimal solutions applied to the real system, and the last $N-1$ elements of $\vsig_m$ are the $N$ predictions (or decision variables) of Problem~\eqref{EqProblem}.
\begin{remark}[Waiting time constraint]\label{rem_UL}
	For the closed-loop system satisfies Assumption~\ref{ass:UL} the following constraint should be added to the restrictions of Problem~\eqref{EqProblem}:
	\begin{align}\label{eq:restriction}
	L_{\sigma(j)}\leq\mid \mathcal{P}[\vsig_m,j]\mid\leq U_{\sigma(j)},~j\in\mathbb Z_{0:N-1}.
	\end{align}
	Note that the $j$-pack $\mathcal{P}[\vsig_m,j]$ -that accounts for all switching signals consecutively identical to the prediction signal $\sigma(j)$- must depend on the past optimal solutions that already enters the closed-loop system, that is the reason to consider the $j$-pack of $\vsig_m$.
\end{remark}

The control law, derived from the application of a receding horizon control policy is given by $\kappa_{MPC}(x)=\sigma^0(0)$, where $\sigma^0(0)$
is the first element of the solution sequence $\vsig^0$ of Problem~\eqref{EqProblem}. 
Therefore, the closed-loop system under the MPC law is described by
\begin{eqnarray}\label{SistCloosed}
x(k+1) = A_{\kappa_{MPC}(x(k))}x(k),
\end{eqnarray}
and the optimal cost function is given by
\begin{eqnarray}\label{costoptimal}
J^0_N(x)= J_N(x,\vsig^0(x)).
\end{eqnarray}

The domain of attraction of Problem~\eqref{EqProblem}, i.e. every state that can be feasible controlled by the SwMPC, is given by $S_N(\Omega,\Sigma)$.

\begin{remark}
	The existence of a SIS implies that the domain of attraction of Problem~\ref{EqProblem}, $S_N(\Omega,\Sigma)$, is not empty, since $\Omega\subseteq S_N(\Omega,\Sigma)$.
\end{remark}
%

%
%\begin{figure}[H]
%	\centering
%	\includegraphics[width=0.65\textwidth]{Figures/horizon.jpg}
%	%	\input{Matlab_fig/layers_xs.tex}
%	\caption{Receding horizon strategy.}
%	\label{fig:receding}
%\end{figure}

%\subsection{Control algorithm} \label{sect_control_algo}
The control algorithm executed at any $k$-th time instant is the following:\\

\begin{algorithm}
	\caption{Switched MPC algorithm} \label{control_algo}
	\begin{algorithmic}[1]
		\Require $N\in\mathbb{N}$, $\X \subset \mathbb{R}^n$ and $\Omega\subseteq \X$
		\State Read $x(k)$
		\State Solve \eqref{EqProblem} subject to \eqref{EqProblem0}-\eqref{EqProblem7}
		\State Inject $\sigma^0(0)$ into the system.
		\State $k \gets k+1$
		\State Go back to 1	
	\end{algorithmic}
\end{algorithm}

The resulting optimization problem is a Mixed Integer Quadratic Programming (MIQP), which can be solved by specific solvers. For the simulations of the present work, Algorithm \ref{control_algo} is implemented in YALMIP, a Toolbox for Modeling and Optimization in MATLAB \cite{Lofberg2004}. The main idea of YALMIP is providing an efficient tool for writing high-level algorithm in MATLAB \cite{Matlab}, while relying on external solvers for the low-level numerical solution of optimization problems. In this work we make use of the Gurobi Optimizer (Version 8.1, Academic License), which in turn relies on a branch-and-bound algorithm to solve MIQP problems \cite{gurobi,Branch-and-bound1960}.\\
%%%%%%%%%%%%%%%%%%%%%%%%%%%%%%%%%%%%%%%%%%%%%%%%%%%%%%%%%%%%%%%%%%%%%%%

%\subsection{Asymptotic stability of the closed-loop system} \label{MainRes}
%%%%%%%%%%%%%%%%%%%%%%%%%%%%%%%%%%%%%%%%%%%%%%%%%%%%%%%%%%%%%%%%%%%%%%%
Before present the stability proof we must assume the following.
\begin{assumption}\label{ass:stabilizable}
	Switched system~\eqref{eq:SistOrig} is stabilizable.
\end{assumption}

\begin{theorem}\label{theo:convergence}
	Let Assumption~\ref{ass:stabilizable} holds, and consider a target set, $\Omega\subset\X$, of cost~\eqref{Cost}, be a SIS for system~\eqref{eq:SistOrig}. Let the initial state $x(i)\in S_N(\Omega,\Sigma)$,  then the set $\Omega$ is asymptotically stable for the closed-loop system~\eqref{SistCloosed}.
\end{theorem}

\begin{proof} 
	Since Assumption~\ref{ass:stabilizable} holds and the initial state at time instant $i$ is such that $x(i)\in S_N(\Omega,\Sigma)$, there is an optimal solution of Problem~\eqref{EqProblem}, given by, \[\vsig^0(x):=\{\sigma^0_0,\sigma^0_1,\ldots,\sigma^0_{N-1} \},\] corresponding with the optimal state sequence,
	\[\vx^0(x):=\{x^0_0, x^0_1,\ldots,x^0_N\},\] such that $x^0_0=x(i)$ and $x^0_N\in \Omega$ (by~\eqref{EqProblem0} and~\eqref{EqProblem7}).
	Then, the optimal cost at time $i$, $J^0_N(x(i))$, is given by,
	\[
	J^0_N(x(i)) = \sum_{j=0}^{N-1} c_{\sigma_j}d_{\Omega}(x^0_j)+ \underbrace{cd_{\Omega}(x^0_N)}_{=0}.
	\]
	Since the target set $\Omega$ is a SIS for switched system~\eqref{eq:SistOrig}, there is a switching signal $\bar{\sigma}\in \Sigma$ such that $\bar{x}=A_{\bar{\sigma}}x^0_N \in \Omega$. 
	Therefore, the sequence, \[\bar{\vsig}=\{\sigma^0_1,\sigma^0_2,\ldots,\sigma^0_{N-1},\bar{\sigma} \},\] is a feasible solution of Problem~\eqref{EqProblem} at time $i+1$, with the feasible state sequence, \[\bar{\vx} =\{x^0_1,x^0_2,\ldots,x^0_N,\bar{x}\}.\]
	This feasible solution corresponds to the following cost:
	\begin{align*}	
	J_N(x(i+1);\bar{\vsig})=\sum_{j=1}^{N-1}c_{\sigma_j}d_{\Omega}(x_j).
	\end{align*}
	Note that, by manipulating $J^0_N(x(i))$ and $J_N(x(i+1);\bar{\vsig})$, we have the following equation:
	\begin{align*}
	J_N(x(i+1);\bar{\vsig}) - J^0_N(x(i)) &=   -c_{\sigma_0}d_{\Omega}(x_0) \\ \nonumber
	&=   -c_{\sigma_0}d_{\Omega}(x(i)).
	\end{align*}
	Therefore, the optimal cost function at time $i+1$, $J^0_N(x(i+1))$, is such that, 
	\begin{align}\label{eq:decreasing_optimal_cost}
	J^0_N(x(i+1)) - J^0_N(x(i))&\leq J_N(x(i+1);\bar{\vsig})-J^0_N(x(i))\\ 
	&= -c_{\sigma_0}d_{\Omega}(x(i)),\nonumber
	\end{align}
	Eq.~\eqref{eq:decreasing_optimal_cost} implies that the positive sequence $J^0_N(x(k))$ decrease when $k\rightarrow\infty$, thus it is easy to prove that $J^0_N(x(k))\rightarrow 0$ when $k\rightarrow\infty$. Therefore, by Eq.~\eqref{eq:decreasing_optimal_cost}, \[d_{\Omega}(x(k))\rightarrow 0,\] when $k\rightarrow \infty$, which means the closed-loop system converges to the desired target set $\Omega$.
\end{proof}

\subsection{Illustrative Example}
Consider the system~\ref{eq:SistOrig} with $\Sigma=\{1,2,3,4\}$, $n = 2$ and 

\begin{equation} \notag
\begin{split}
A_1=\left[ \begin{array}{c c}
1.5 & 0  \\
0 & -0.8 \end{array} \right] , ~~~ A_2=1.1 R(\frac{2\pi}{5}).	 
\end{split}
\end{equation}
\begin{equation} \notag
\begin{split}
A_3= 1.05 R(\frac{2\pi}{5}-1) , ~~~ A_4=\left[ \begin{array}{c c}
-1.2 & 0 \\
1 & 1.3 \end{array} \right].	 
\end{split}
\end{equation}
where $R(\cdot)$ is the rotation matrix. This is an example of a stabilizable switched system settled for 4 not Schur sub-systems $A_1$, $A_2$, $A_3$ and $A_4$ (see~\cite[Example 2]{FiacchiniAUT14}). 
\begin{figure}[H]
	\centering
	\includegraphics[width=0.65\textwidth]{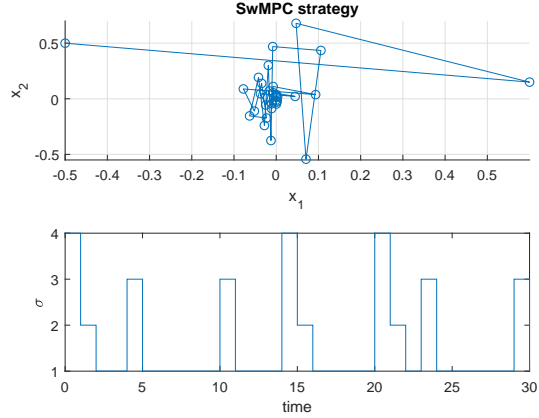}
	\caption{Optimal trajectory (top figure). Optimal switching signal sequence (down figure).}
	\label{fig:SwMPCexample}
\end{figure}
Figure~\ref{fig:SwMPCexample} shows the application of the proposed SwMPC with a prediction horizon $N=15$, a target set is $\Omega=\{\vec{0}\}$ (particular case of SIS). The state constraint $\{x:-10\le x_1\le 10;-10\le x_2\le 10\}$, and simulation time of $T=30$. The initial state $x_0=[-0.5,0.5]$ belongs to the $N$-controllable set $S_{N}(\Omega,\Sigma)$.

%%%%%%%%%%%%%%%%%%%%%%%%%%%%%%%%%%%%%%%%%%%%%%%%%%%%%%%%%%%%%%%%%%%%%%%%%%%%%%%%%%%%%

\section{Applications to biomedical problems}\label{Sec_Application}
This section presents two particular applications of the SwMPC previously discussed.

\subsection{Drug schedules to mitigating viral escape}\label{Sec:HIV}
We focus on the problem of treatment scheduling to minimize the adverse effects of virus mutation in acute and chronic infections. 
Acute infections are caused by virus that grows out but it is cleared in a short period. On the other hand, in chronic viral infections the virus grows slowly promoting persistence.
In both scenarios, the main problematic is given by the rise to drug resistance. 
For this work we focus on the virus mutation treatment problem. For this reason we use a model for mutation dynamics that is simple enough to allow control analysis and optimization of treatment switching \cite{hernandez2011discrete}. The model is given by:
\begin{align}\label{eq:simplermodel}
\dot{V}_i(t)=& \rho_{i,\sigma(t)}V_i(t)-\delta V_i(t)+\sum_{i\not = j}\mu m_{i,j}V_j(t)
\end{align}
where $\mu$ is a small parameter representing the mutation rate, $\delta$ is the death or decay rate of the virus and $m_{i,j}\in\{0,1\}$ represents the genetic connections between genotypes, that is, $m_{i,j}=1$ if and only if
it is possible for genotype $j$ to mutate into genotype $i$. Equation \eqref{eq:simplermodel} can be rewritten as
\begin{align}\label{eq:simplermodelVector}
\dot{V}(t)= &(R_{\sigma(t)} - \delta I)V(t)+\mu M V(t)
\end{align}
where $M:=[m_{i,j}]$ and $R_{\sigma(t)}:=diag\{\rho_{i,\sigma(t)}\}$, and every element of $V(t)$ is a particular genotype.

We take a model with four genetic variants and two possible drug therapies, as shown in Figure~\ref{fig:genotipos}. 

\begin{figure}[H]
	\centering
	\includegraphics[width=0.35\textwidth]{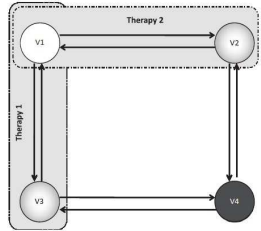}
	\caption{The virus $V_1$ is susceptible to both therapies. $V_2$ is susceptible to therapy 2 while $V_3$ is susceptible to therapy 1. There is a highly resistant genotype ($V_4$) which is resistant to therapy 1 and 2.}
	\label{fig:genotipos}
\end{figure}
Figure~\ref{fig:genotipos} shows a mutation graph that is symmetric and circular, only connections: $V_1(t)\leftrightarrow V_2(t)$, $V_2(t)\leftrightarrow V_4(t)$, $V_4(t)\leftrightarrow V_3(t)$, $V_3(t)\leftrightarrow V_1(t)$ are possible. This leads to the mutation matrix:
\begin{eqnarray}\label{eq:conectionM} 
M=\left[ \begin{array}{cccc}
0 & 1 & 1 & 0 \\ 
1 & 0 & 0 & 1 \\
1 & 0 & 0 & 1 \\
0 & 1 & 1 & 0 \\
\end{array} \right]
\end{eqnarray}

\subsubsection{Numerical results}
This section provided some simulations results to schedule the treatment for viral mutation problems. First, we will define therapeutic strategies that are recommended in clinics:
\begin{itemize}
	\item[i.] The \emph{switching on virologic failure} (VF) strategy, recommended by \cite{AIDSInfo13}, introduce a new regimen when there is detectable viremia (Virus RNA $>1000~copies/ml$) and a drug-resistant genotype is identified. 
	\item[ii.] The SWATCH approach, recommended by \cite{Martinez08}, is based on the possibility of preempt virologic rebound; this strategy reduces the accumulating drug-resistant genotypes by alternating between the two regimes every three months while viral load is suppressed. 
\end{itemize}

System \eqref{eq:simplermodelVector} is described in discrete-time with a regular treatment interval $\tau=28~days$; during this time interval the treatment is considered to be fixed. If $k\in\N$ denotes the number of intervals, equation \eqref{eq:simplermodelVector} can be described by the following discrete-time switched linear system~\eqref{eq:SistOrig}, where $x(k)=x(k\tau)$ is the sampled state and $A_{\sigma}=e^{(R_{\sigma}-\delta I +\mu	M)\tau}$. The state is constrained to $x(k)\in\X:=\R_{\geq 0}$, and $\sigma(k)\in\{1,2\}$ for all $k\in\Z_{\ge 0}$.

Viral mutation rates are about $\mu=10^{-4}$ and the connection matrix by~\eqref{eq:conectionM}. We consider the initial condition 
\begin{align}\notag
V_1(0)&=1000\mbox{ copies/ml},~~V_2(0)=\mu V_1(0),\notag\\
V_3(0)&=\mu V_1(0),~~V_4(0))=\mu V_2(0)+\mu V_3(0),
\end{align}
and the viral clearance rate is $\delta=0.24/day$, which corresponds to a half life less than 3 days. As we mentioned before, the decision time is $\tau=28~\mbox{days}$, for a period of $T=336~\mbox{days}$. 

The rates of the viral replication under treatment $\sigma$, $R_{\sigma}$, can define several cases, among them the chronic and the acute infection scenarios. In \cite[Chapter~7]{Vargasbook19} several chronic scenarios for replication rates are presented; some ideal cases describe a complete symmetry between genotypes $V_2$ and $V_3$: therapy 1 inhibits $V_3$ with the same intensity as therapy 2 inhibits $V_2$. More detailed models also include asymmetry in the genetic tree with more complex structure than a simple cycle. Another scenarios show an asymmetry for replication rates in genotypes $V_2$ and $V_3$, although both therapies induce the same replication in genotypes $V_1$ and $V_4$ (Scenario 1, Table~\ref{tab:replicationrates}). Another scenarios are more realistic, each genotype experience different dynamics to a new treatment. However, all these scenarios represent chronic infections. In this work we proposed a new one corresponding to an acute infection (e.g. influenza), characterized by a rapid development of the viral load, which however may be cleared in short time (Scenario 2, Table~\ref{tab:replicationrates}). 

\begin{table}
	\centering
	\begin{tabular}{p{1cm} p{1.8cm} p{1.8cm} p{1.8cm} p{1.8cm}}
		
		Scenario & \centering $V_1$ & \centering$V_2$ & \centering $V_3$ &  ~~~~~~~$V_4$\\
		\hline
		\centering$1$  & $\rho_{1,1}=0.05$ &$\rho_{2,1}=0.28$ & $\rho_{3,1}=0.01$ & $\rho_{4,1}=0.27$ \\
		& $\rho_{1,2}=0.05$ &$\rho_{2,2}=0.20$ & $\rho_{3,2}=0.25$ & $\rho_{4,2}=0.27$ \\
		
		\hline
		\centering$2$  & $\rho_{1,1}=0.05$ &$\rho_{2,1}=0.40$ & $\rho_{3,1}=0.05$ & $\rho_{4,1}=0.23$ \\
		& $\rho_{1,2}=0.05$ &$\rho_{2,2}=0.05$ & $\rho_{3,2}=0.40$ & $\rho_{4,2}=0.23$ \\
		\hline
	\end{tabular}
	\caption{Replication rates ($R_{\sigma}$) for viral variants and therapy combinations. Scenarios 1 and 2 represent chronic and acute infections respectively.}
	\label{tab:replicationrates}
\end{table}

The \emph{total viral load} at time instant $k$, $V_{total}(k)$,  is defined by $V_{total}(k)=\sum_{i=1}^4V_i(k)$, where $V_i(k)$ is the viral load of genotype or strain $i$ at time instant $k$.
%%%%%%%%%%%%%%%%%%%%%%%%%%%%%%%%%%%%%%%%%%%%%%%%%%%%%%%%%%%%%%%%%
%%%%%%%%%%%%%%%%%%%%%%%%%%%%%%%%%%%%%%%%%%%%%%%%%%%%%%%%%%%%%%%%%

The therapeutic strategies that we are going to test are the switching on virologic failure (VF) and the SWATCH approach, both presented above.
On the other hand, for the presented scenarios the optimal solution will be computed by the ``brute force'' approach, which analyzes the best numerical solution of all possible combinations for therapies $1$ and $2$ with decision time $\tau=28~days$ for a period of $T=336~days$. That is, $2^{T/\tau}$ possible treatment combinations are evaluated and the sequence of treatments that gives
the least amount of total viral load over the whole period of time is chosen (i.e., the one that minimizes the cost~\eqref{eq:CostOrigin}).
Notice that this approach has more computational complexity as the period of time is incremented or the treatment interval is reduced, making it a not implementable optimization.

Simulations for the \emph{switching on virologic failure} show that initially the total viral load drops rapidly for the chronic infection. However, the appearance of resistant genotype will drive a virologic failure after 200 days making a new therapy necessary. 
The acute infection has a very different behavior; unlike the others cases, the system can be stabilized, i.e. the total viral load can be driven to undetectable levels ($V_{total}(T)\le 50~~\mbox{copies/ml}$). However, since the therapy changes regimen by an exceeding of an upper bound, the total viral load does not reach its minimum values and shows an oscillating behavior. 
The SWATCH strategy shows - as previously highlighted by \cite{hernandez2011discrete} - a better performance than the switching on virologic failure: a lower concentration in the total viral load over the year for chronic infections, while the viral population is cleared in acute infection.
The optimal solution for the chronic infections is given by ``brute force'' approach, in order to compare performance of the proposal with the best possible result.
In the chronic infection scenario there is always a viral escape, because the high resistance genotype rises with resistance for the two regimens. 

\subsubsection{MPC-based scheduling method}
The same scenarios studied above will be tackled by the MPC proposed in this work. 
A prediction horizon of $N=5$ is considered - equivalent to $5\tau ~\mbox{days}$ - with the decision time $\tau=28~\mbox{days}$ for a period of $T=336~\mbox{days}$. The objective of the controller is to drive the total viral load to undetectable levels ($V_{total}(T)<50~\mbox{copies/ml}$). Since the objective can not be maintained over time for chronic infections due to the promoting persistence of high resistance genotypes it is expected that the proposed strategy delays the viral escape time. 
\begin{figure}
	\centering
	\includegraphics[width=0.75\textwidth]{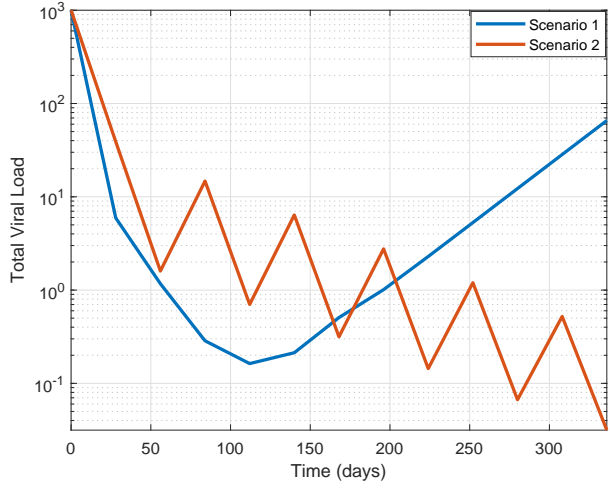}
	\caption{Total viral load for chronic and acute scenarios under the SwMPC strategy.}
	\label{fig:SwMPC_1}
\end{figure}
Figure~\ref{fig:SwMPC_1} shows  the two scenarios under the proposed SwMPC. As it can be seen, the controller suppresses the viral load, for both cases. In the chronic infection scenario the total viral load is maintained below the virologic failure levels, while in acute infections it is completely cleared. 
Figure~\ref{fig:SwMPCSc3}, on the other hand, shows the behavior of all genotypes only for Scenario 2, together with the switching sequence provided by the SwMPC. The sequence is not intuitive at all, since therapy 1 is used only three times throughout the year of treatment, in the third, sixth and eleventh month.
\begin{figure}[H]
	\centering
	\includegraphics[width=0.75\textwidth]{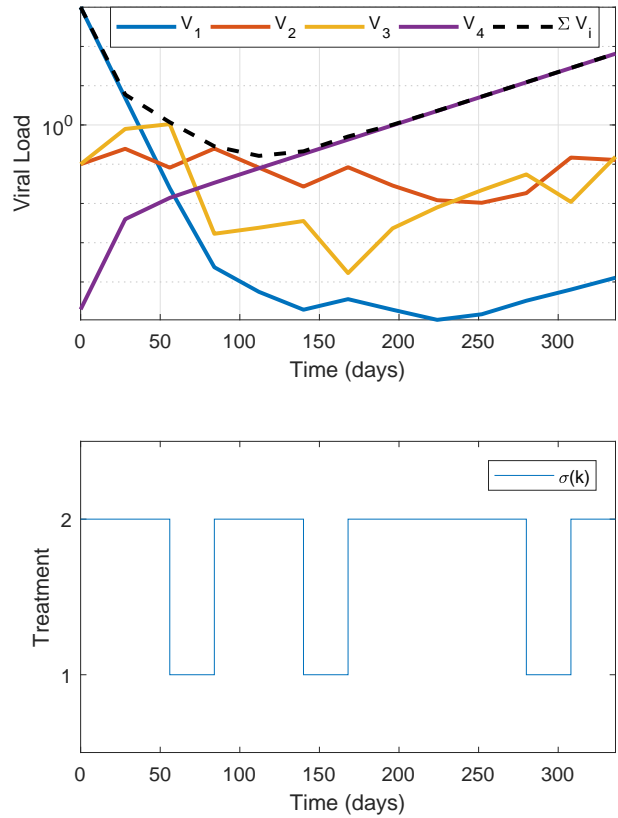}
	\caption{Viral load for Scenario 2 and treatment by the SwMPC.}
	\label{fig:SwMPCSc3}
\end{figure}
\begin{remark}
	It is important to highlight that in chronic infection scenarios, where the system cannot be stabilized, the viral escape cannot be avoided. However, the simulation results suggest that the proposed MPC delays the escape time, which is considerably beneficial in this context.
\end{remark}

In chronic infections the virus grows slowly, promoting this way persistence and producing an unavoidable viral escape. This fact makes the switched system that models the infection dynamic essentially non stabilizable and, so, it will not be possible to drive the system to the undetectable virus zone and keep it there indefinitely.
So, the following index is proposed to compare the existing strategies with the proposal:
\begin{align}\label{Eq:Index}
\mathcal{I}_T = \sum_{k=0}^{T}V_{total}(k),
\end{align}
where $V_{total}(k)$ is the \emph{total viral load} at time instant $k$

The best performance is obtained by the optimal solution computed by ``brute force'' approach, according to Table~\ref{Tab:performance}.  
\begin{table}[H]
	\centering
	\begin{tabular}{p{1cm} p{1.5cm} p{1.5cm} p{1.5cm} p{1.5cm}}
		
		\centering	Scenario &  \centering SWATCH & \centering VF & OPTIMAL &  SwMPC\\
		\hline 
		
		\centering$1$  & \centering$1587.1$ & \centering$5277.9$ &$1108.4$ & $1123.3$  \vspace{.3cm}\\

		\centering$2$  & \centering$1175.6$ & \centering$12075.0$ &$1067.4$ & $1067.6$  \\
		
		\hline
	\end{tabular}
	\caption{Performance index $\mathcal{I}_T$ for all strategies.}
	\label{Tab:performance}
\end{table}
The indexes in Table~\ref{Tab:performance} reveal that the proactive switching strategies may outperform the ``switched on virologic failure'' strategy, as it was previously stated in \cite{hernandez2011discrete}. 
Nevertheless, the proposed SwMPC provides better results than SWATCH treatment, and exhibits almost the same performance than the optimal solution, in all cases, which is a result to be highlighted considering that the MPC is an implementable strategy, which is robust to model-plant mismatches, explicitly considers constraints and has a low computational burden.

In the next section we seek to address a controller synthesis problem for cancer treatment. 

\subsection{Drug schedules for cancer treatment}\label{Sec:Cancer}
Cancer cells have the potential to develops resistance to therapies. This is a complex phenomenon influence by multifaceted mechanisms. Mathematical models are important tools to further understand cancer cell dynamics and the respective outcomes during therapies \cite{Kumar19}. In particular model the Triple Negative breast cancer cell line, HCC1143, with response to different drugs was obtained by a linear time-invariant system identified in \cite{ChapmanCDC16} and \cite{RisomNat18}. 
The analysis of drug schedules in terms of switched systems was previously study in \cite{ChapmanCDC18}, with the assumptions that the response to a drug applied at time $k$ does not depend on the drugs applied previously and there are some drugs that can shrink the live cancer cell population, but are very toxic to healthy cells, and other drugs that only slow the growth of the live cancer cell population, but are less toxic to healthy cells.\\

Cancer cell population can be partitioned into a finite number of observable traits that arises from the synthesis of proteins that has important implications for drug response, called phenotypic states; e.g., see \cite{GoldmanNat15}. The model of drug-treated cancer cell population is given by the switched system~\eqref{eq:SistOrig}, with $x(k) \in \R_+^n$, the non-negative cell type vector, the elements of the sate $x(k)$ are the number of live cells in the phenotypic state $i=1,2,\dots,n$, the switched signal $\sigma$ belongs to a set of drugs $\Sigma$, and the matrix, $A_{\sigma}$, takes the form,
\begin{eqnarray}\label{matrizAsigma} 
A_{\sigma}=\left[ \begin{array}{cccc}
\alpha_1 & \rho_{21} & \dots& \rho_{n1} \\ 
\rho_{12} & \alpha_2& \dots& \rho_{n2}\\
\vdots& \vdots& \ddots & \vdots\\
\rho_{1n}& \rho_{2n}& \dots & \alpha_n\\
\end{array} \right],
\end{eqnarray}
where $\alpha_i:=\rho_i-\rho_{iD}-\sum_{s=1,s\ne i}^n \rho_{is}$; for all $i=1,\dots,n$. The parameters $\rho_i=\rho_i(\sigma)$, $\rho_{iD}=\rho_{iD}(\sigma)$, and $\rho_{ij}=\rho_{ij}(\sigma)$ are defined in Table~\ref{Tab:drugforCancer} (see~\cite{ChapmanCDC18}).  

\begin{table}[H]
	\centering
	
	\begin{tabular}{p{2.5cm} p{2.2cm} p{6.5cm}}
		\hline
		Symbol &   Name &  Definition \\
		\hline 
		$\rho_i=\rho_i(\sigma)$  &  Division parameter &   Average ratio of the number of live cells in phenotypic state $i$ at time $k+1$, derived from their own kind, to the number of live cells in phenotypic state $i$ at time $k$ \\
		\hline
		$\rho_{iD}=\rho_{iD}(\sigma)$  &  Death gain &  Average proportion of live cells in phenotypic state $i$ at time $k$ that begin to die, or are dead, by time $k + 1$\\ 
		\hline		
		$\rho_{ij}=\rho_{ij}(\sigma)$  &  Transition gain &  Average proportion of live cells in phenotypic state $i$ at time $k$ that transition to phenotypic state $j$ by time $k + 1$; $i\neq  j$\\
		\hline
	\end{tabular}
	\caption{Dynamics parameters for drug $\sigma$ (taken from \cite{ChapmanCDC18}).}
	\label{Tab:drugforCancer}
\end{table}
Every treatment is subject to toxicity issues or to the onset of resistance. It is therefore important to ensure a particular waiting time for every drug (Assumption~\ref{ass:UL}). In addition, the schedule of treatments are applied by cycles; a cyclic $m$ is a sequence of drugs defined by,
\begin{align}\label{eq_cyclic}
\{(\sigma_{1m},h_{1m}),(\sigma_{2m},h_{2m}),\dots,(\sigma_{qm},h_{qm}) \}_{m=1}^{\infty},
\end{align}
such that $\sigma_{im}\in\Sigma$ is the $i^{th}$ drug applied in cyclic $m$ and $h_{im}$ is the waiting time for drug $\sigma_{im}$, such that $L_{\sigma_{im}}\le h_{im} \le U_{\sigma_{im}}$, and $\cup_{i=1}^q \sigma_{im}=\Sigma$ for each $m$. The waiting times are design such that the treatment regimen shrinks the live cancer cell population over time, while limiting the toxicity to normal cells or avoiding the onset of drug resistance. 

\subsubsection{Numerical results}
Drug treatment for Triple Negative breast cancer cell line belong to the set $\Sigma =\{Trametinib+BEZ235, BEZ235, Trametinib\}$ (see Remark~\eqref{rem_drugs}), the cell phenotype measurements were recorded every 12 hours over a 72-hour horizon, and the dynamics matrix was estimated from the data for $n = 2$ phenotypic states \cite{RisomNat18}.

The dynamic matrices, $A_{\sigma}$, for the live cancer system are given by,
\begin{align*}
A_P=\left[ \begin{array}{cc}
0.755 & 0.081  \\ 
0.169 & 0.843\\
\end{array} \right],~ A_B=\left[ \begin{array}{cc}
0.896 & 0 \\ 
0.186 & 1.083\\
\end{array} \right], ~A_T=\left[ \begin{array}{cc}
1.030 & 0.231  \\ 
0.022 & 0.821\\
\end{array} \right],
\end{align*}
where drug $P$ is $Trametinib+BEZ235$, drug $B$ is $BEZ235$, and drug $T$ is $Trametinib$. The initial condition, $x_0 :=[220; 612]^T$, is the estimated number of live cells in each phenotypic state at time zero averaged over fifteen samples (see~\cite{ChapmanCDC18}). 
The waiting times of Assumption~\ref{ass:UL} are given by $L_i := 2$ (1 day) for all $i\in\Sigma$, $U_P := 4$ (2 days), $U_B := 8$ (4 days), and $U_T := 6$ (3 days).
\begin{remark}\label{rem_drugs}
	Drug $Trametinib+BEZ235$ can shrink the live cancer cell population but is highly toxic to normal cells, and is related with the \emph{stable sub-system}, $A_P$; drugs $BEZ235$ and $Trametinib$ are less toxic to normal cells but can only reduce the growth rate of live cancel population, and are related with the \emph{unstable sub-systems}, $A_B$ and $A_T$, respectively. 
\end{remark}

The cyclic optimal schedule of the proposed SwMPC is given by, 
\begin{align}\label{eq_solution}
\{(P,4),(T,2),(B,2) \}_{m=1}^{\infty},
\end{align}
which means that for every cyclic $m$, drug $P$ is applied $4$ consecutive times, then drug $T$ is applied $2$ consecutive times followed by drug $P$ applied $2$ consecutive times as well. 
\begin{remark}\label{rem_solution}
	Cyclic~\eqref{eq_solution} is a trivial solution to stabilize the switched system, because according to Eq.~\eqref{eq_solution} the unique stable sub-system, $A_P$, is applied to  the switched system~\eqref{eq:SistOrig} its maximum waiting time (saturating the application of this drug), and the unstable sub-systems, $A_T$ and $A_B$, are applied, by restrictions, just their minimum waiting times (in order to return the employment of the stable sub-system).
\end{remark}

In what follows, we penalize by cost the \emph{consecutive} use of each treatment instead of the constraint of maximum waiting time $U_{\sigma}$. Consider the following modified cost function
\begin{align}\label{eq_CostToxicity}
\bar J_N(x;\vsig):=  J_N(x;\vsig) + \sum\limits_{j=0}^{N-1} b_{\sigma(j)}h^{\sigma(j)}(j),
\end{align}
where $J_N(x;\vsig)$ is given by Eq.~\eqref{Cost}, $b_{\sigma}$ is a positive weight that penalize the \emph{consecutive} use of treatment $\sigma$, and,
\begin{align*}
h^{\sigma(j)}(j)&=\mid \mathcal{P}[\vsig,j] \mid^2
\end{align*}
is the square cardinal of the $j$-pack $\mathcal{P}[\vsig,j]$, which accounts all switching signals consecutively identical to the prediction signal $\sigma(j)$, therefore term $h^{\sigma}(\cdot)$ penalizes the consecutive use of signal $\sigma$.\\

The simulations will be performed with the same conditions presented before, but regardless of the cycles and the maximum waiting times $U_{\sigma}$ of Assumption~\ref{ass:UL}, in order to extend the constraints of the optimization problem and therefore analyze a less limited result. Illustratively, we will consider three different cases based on the concept that there are some drugs that are very toxic to healthy cells and other drugs that are less toxic to healthy cells.
\begin{itemize}
	\item[] Case 1. $b_P = b_T = b_B$. The consecutive use of every treatment is equally penalized by the cost~\eqref{eq_CostToxicity}.
	\item[] Case 2. $b_P > b_T = b_B$. The consecutive use of drug $P$ is slightly more penalized than the consecutive use of drugs $T$ or $B$.
	\item[] Case 3. $b_P >> b_T > b_B$. The consecutive use of drug $P$ is a sight more penalized than the consecutive use of drug $T$, and this last one is slightly more penalized than the consecutive use of drug $B$.
\end{itemize}

%\begin{table}[H]
%	\centering
%	\begin{tabular}{p{1cm} p{1cm} p{1cm} p{1cm}}
%		
%		\centering	 &   $b_B$ & $b_T$ & $b_P$\\
%		\hline 
%		\centering$Case 1$  & $1$ & $1$  & $1$  \\
%		\centering$Case 2$  & $1$ & $1$  & $10$  \\
%		
%		\centering$Case 3$  & $1$ & $150$  & $200$ \\
%	
%		\hline
%	\end{tabular}
%	\caption{Several weights for the consecutive use of each treatment $\sigma$.}
%	\label{Tab:weight}
%\end{table}
%
The first case is shown in Figure~\ref{fig:cancerCase1}. The hypothetical Case 1, considers that each treatment has the same ``toxicity level'', therefore the optimal schedule (at the beginning) is the application of the most effective treatment, $P$. 
Although each drug is equally penalized, cost~\eqref{eq_CostToxicity} also penalizes the consecutive use of any of them, so after $8$ times applied treatment $P$, there is a switch to drug $B$, which is applied the minimum waiting time, $L_B=2$, finally returning to drug $P$. 
%
%%%%%%%%%%%%%%%%%%%%%%%%%%%%%%%%%%%%%%%
Figure~\ref{fig:cancerCase2} shows the result for Case 2. The simulation represents a similar scenario to the first case, nevertheless, the slightly greater weight $b_P$ results in a less number of consecutive times that treatment $P$ is applied to the system than in Case 1.
%
%%%%%%%%%%%%%%%%%%%%%%%%%%%%%%%%%%%%%%
Figure~\ref{fig:cancerCase3} represents simulations for the third case showing a complete different scenario; the population of phenotype cells decays slower than the first cases because the consecutive use of the most effective treatments are highly penalized.
%
%%%%%%%%%%%%%%%%%%%%%%%%%%%%%%%%%%%%%%%%%%%%%%%%%%%%%%%%%%%%%%%%%%%%%%%%%%%%%%%%%%%%%
\begin{figure}
	\centering
	\includegraphics[width=0.75\textwidth]{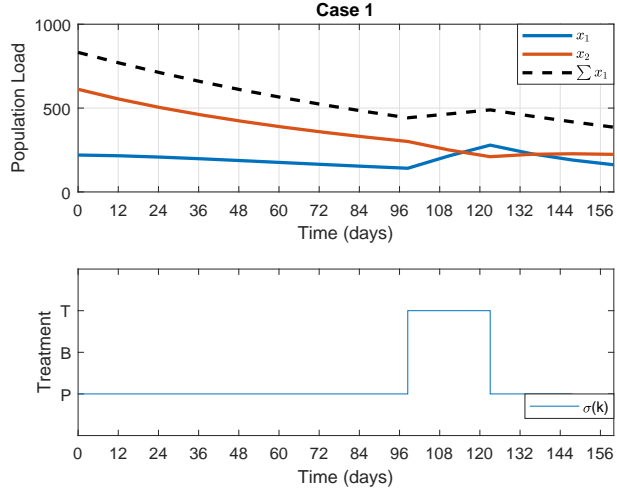}
	\caption{Phenotype cell measurement and proposed treatment for Case 1.}
	\label{fig:cancerCase1}
\end{figure}
\begin{figure}
	\centering
	\includegraphics[width=0.75\textwidth]{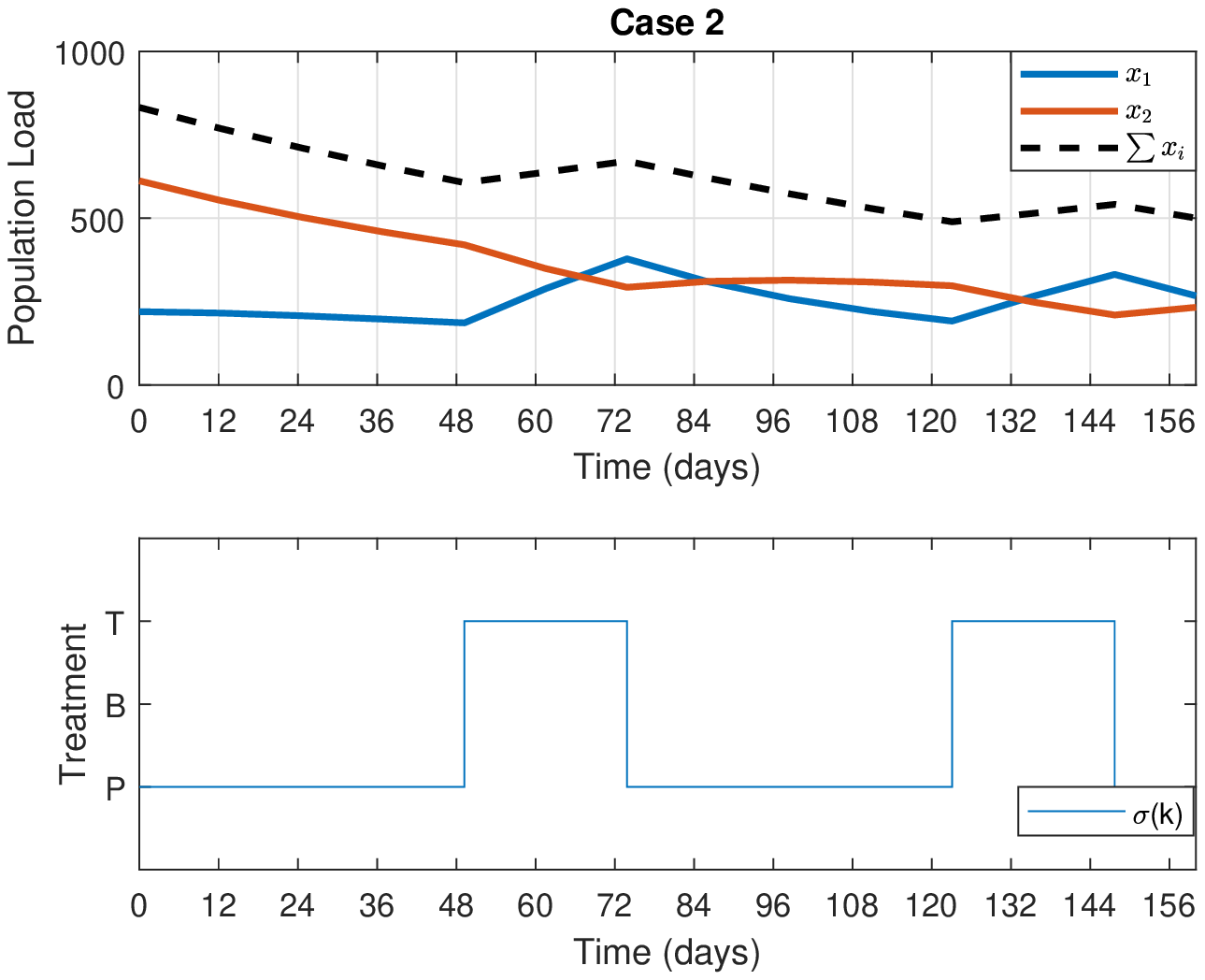}
	\caption{Phenotype cell measurement and proposed treatment for Case 2.}
	\label{fig:cancerCase2}
\end{figure}
\begin{figure}
	\centering
	\includegraphics[width=0.75\textwidth]{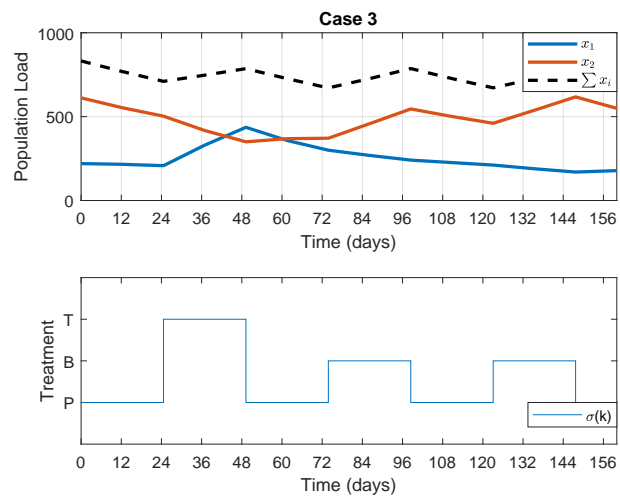}
	\caption{Phenotype cell measurement and proposed treatment for Case 3.}
	\label{fig:cancerCase3}
\end{figure}

\begin{remark}
	Generally, waiting time constraints related to toxicity or to the onset of resistance (regardless the application) are satisfied by restrictions, incorporating computational complexity to the optimization problem, and making them some times a difficult, if not impossible, task. We have shown in the last simulations that the waiting times can be cornered by minimizing a particular function cost instead of adding hard constraints to the problem, which is understood -by the control point of view- as a benefit in the implementation of the optimization problem.
\end{remark}

%Figure~\ref{fig:cancerNormPop} shows the stabilization of the normalized population size, $\frac{\|x(k)\|_1}{ \|x_0\|_1 }$ for an extended simulation time. 
%%
%\begin{figure}[H]
%	\centering
%	\includegraphics[width=0.5\textwidth]{Figures/cancerNormPop.eps}
%	%	\input{Matlab_fig/layers_xs.tex}
%	\caption{Normalized population size $\frac{\|x(k)\|_1}{ \|x_0\|_1 }$, for the cancel cell population $x(k)$ treated with the proposal drug schedule.}
%	\label{fig:cancerNormPop}
%\end{figure}
%
\section{Conclusions}\label{Sec:conclusion}

This paper proposed a proper set-based MPC for discrete-time switched system. The formulation drives the controlled states to a given region of the sate space, which is defined as an invariant set for the switched system. 
The proposal was applied to a simplified viral mutation model, proving that it attenuates the effect of the viral mutation in several scenarios containing the chronic and acute infection cases. The strategy was compared with some basic viral mutation treatments and the optimal solution showing that the adverse effects of virus mutation in acute and chronic infection are minimized. 
%
%In acute infections the proposed controller cleared the total virus load in a short period, close to the optimal solution, and in chronic infection infections the results suggest that the proposed MPC extends the time to viral escape. 
%
The proposal was applied to schedule an optimal solution for cancer treatment as well. The possibility of satisfying the maximum waiting times between drugs inputs related to toxicity or the onset of resistance for some drugs by minimizing a particular function cost was analyzed. The results showed that the proposal fulfilled different waiting times for each drug by considering different weights that penalized the consecutive use of each treatment.

\bibliographystyle{elsarticle-num} 
\bibliography{bib_yalmip,bib_aanderson}

\end{document}